\documentclass[12pt,a4paper]{article}
\usepackage{amsmath}
\usepackage{bbding}
\usepackage{cases}
\usepackage{amsfonts}
\usepackage{mathrsfs}
\usepackage{amssymb}
\usepackage{graphicx}
\usepackage{enumerate}
\usepackage{datetime}
\usepackage{color}
\usepackage [top=0.8in, bottom=1in, left=2cm, right=2cm]{geometry}

\allowdisplaybreaks

\newtheorem{theorem}{Theorem}[section]
\newtheorem{definition}[theorem]{Definition}

\newtheorem{lemma}[theorem]{Lemma}
\newtheorem{corollary}[theorem]{Corollary}
\newtheorem{proposition}[theorem]{Proposition}
\newtheorem{example}[theorem]{Example}
\newtheorem{remark}[theorem]{Remark}
\newenvironment{proof}{\noindent {\bf Proof. \ }}{$\blacksquare$\vspace{2ex}}

\newcommand{\GK}{\operatorname{GKdim}}
\newcommand{\Aut}{\operatorname{Aut}}
\newcommand{\car}{\operatorname{card}}
\newcommand{\id}{\operatorname{id}}

\newcommand{\auto}{automorphism}

\newcommand{\NN}{\mathbb{N}}
\newcommand{\RR}{\mathbb{R}}
\newcommand{\ZZ}{\mathbb{Z}}

\newcommand{\QQ}{\mathbb{Q}}

\newcommand{\fd}{finite dimensional}
\newcommand{\gk}{Gelfand-Kirillov dimension}
\newcommand{\dd}{differential difference}

\newcommand{\D}{\mathcal{D}}
\newcommand{\s}{\mathcal{S}}

\usepackage{hyperref}
\begin{document}
\title{Gelfand-Kirillov dimension of differential difference algebras}
\author{\small Yang Zhang and
Xiangui Zhao\\
\small Department of Mathematics, University of Manitoba\\
\small Winnipeg, MB, R3T 2N2, Canada\\
\small yang.zhang@umanitoba.ca, \ \ xian.zhao@umanitoba.ca
}
\date{}

\maketitle

\begin{abstract}
Differential difference algebras were introduced by Mansfield and Szanto, which arose naturally from differential difference equations. In this paper, we investigate the Gelfand-Kirillov dimension of differential difference algebras. We give a lower bound of the Gelfand-Kirillov dimension of a differential difference algebra and a sufficient condition under which the lower bound is reached; we also find an upper bound of this Gelfand-Kirillov dimension under some specific conditions and construct an example to show that this upper bound can not be sharpened any more.
\end{abstract}

\textbf{Keywords}: differential difference algebra, Gelfand-Kirillov dimension, Ore extension.

\textbf{MSC 2010}: 16P90, 16S36

\maketitle

\section{Introduction}
Differential difference algebras (Definition \ref{Def_DD-alg})
arose naturally from differential difference equations \cite{courtois2000efficient, mansfield2003elimination}.
The class of \dd\ algebras contains several well-known classes of noncommutative algebras, for example,  commutative polynomial algebras, quantum planes, and skew polynomial algebras of derivation (or automorphism) type.
Rougly speaking, a \dd\ algebra is a noncommutative  polynomial ring (over an algebra) with two sets $D$ and $S$ of indeterminates, where $D$ originally stands for differential operators and $S$ originally stands for shift operators (the difference operators can be derived from $S$). Operators in $D$ ($S$, respectively) commute with each other, but not with those in $S$ ($D$, respectively). The exact definition is given in Section 2.

Let $k$ be a field and $A$ be a unital associative $k$-algebra. The Gelfand-Kirillov dimension of $A$ is defined as
$$\GK(A)=\sup_V\overline{\lim_{n\to\infty}}\log_n \dim_k(V^n)$$
where the supremum is taken over all finite dimensional subspaces $V$ of $A$.
The Gelfand-Kirillov dimension is a very useful and powerful tool for investigating noncommutative algebras.
Basic properties of  Gelfand-Kirillov dimension can be found in \cite{Krause-Lenagan_2000}.

There have been a number of results concerning  Gelfand-Kirillov dimensions of algebras with derivations and/or automorphisms,
for example, the Gelfand-Kirillov dimension of Ore extensions of derivation type \cite{lorenz1982gelfand},
of Ore extensions of automorphism type \cite{Huh-fand_1996},
 of PBW-extensions \cite{matczuk1988gelfand}, and of skew polynomial extensions \cite{zhang1997note}.

Assume that $R$ is a unital associative $k$-algebra. Let $\sigma$ be an automorphism of $R$,  $\delta$ be a $\sigma$-derivation of $R$ and  $A=R[x;\sigma,\delta]$ be an Ore extension over $R$.
It was shown in \cite{Huh-fand_1996} that
\begin{eqnarray}\label{eqn_inequality}
\GK(A)\geq \GK(R)+1;
\end{eqnarray}
the inequality becomes an equality provided that the following condition holds:
\begin{itemize}
\item[($*$)] each \fd\ subspace $U$ of $R$ is contained in a \fd\ subspace $V$
such that $\sigma(V)\subseteq V$ and $\delta(V)\subseteq V^p$ for some $p\geq 1$.
\end{itemize}

In this paper we investigate the Gelfand-Kirillov dimension of  \dd\ algebras.
We show that Inequality (\ref{eqn_inequality}) of the Gelfand-Kirillov dimension of an Ore extension
can be extended to \dd\ algebras, that is,
we get a lower bound of the Gelfand-Kirillov dimensions of \dd\ algebras.
However, owing to the noncommutativity of indeterminates of differential difference algebras,  even under conditions similar to $(*)$,
the equality does not hold for \dd\ algebras in general.
We find an upper bound for the Gelfand-Kirillov dimension of a \dd\ algebra satisfying (an analogues of) condition $(*)$, and give a sufficient condition under which the lower bound is reached.
We also construct an example to show that the upper bound we obtain cannot be sharpened any more.

This paper is organized as follows. Definition and examples of \dd\ algebras are given in Section 2.
The Gelfand-Kirillov dimension of \dd\ algebras is investigated in Section 3.

\section{Preliminaries}

Throughout this paper, we assume that  $k$  is a field and all algebras  are unital  associative $k$-algebras.
Denote the set of $k$-algebra homomorphisms of algebra $A$ by $\Aut(A)$. If $\sigma\in \Aut(A)$, then a mapping $\delta$ on $A$ is called a $\sigma$-derivation provided that, for any $a,b\in A$ and $c\in k$, $\delta(ca+b)=c\delta(a)+\delta(b)$ and
$\delta(ab)=a\delta(b)+\delta(a)b$.
Particularly, if $\sigma=\id$ then $\delta$ is called a derivation on $A$.

First we recall the definition of \dd\ algebras, which was introduced by Mansfield and Szanto \cite{mansfield2003elimination} with some discussions of Gr\"obner bases.

\begin{definition}\label{Def_DD-alg}(cf., \cite{mansfield2003elimination})
An algebra $A$ is called \emph{a \dd\ algebra} of type $(m,n)$, $m,n\geq 1$, over a subalgebra $R\subseteq A$ if there exist elements
$S_1,\ldots,S_m,D_1,\ldots,D_n$ in $A$ such that
\begin{enumerate}[(i)]
\item the set
$\{S^{\alpha}D^{\beta}:\alpha\in \NN^m, \beta\in \NN^n\}$
 forms a basis for $A$ as a free left $R$-module.
\item $D_ir=rD_i+\delta_i(r)$ for any $1\leq i\leq n$ and $r\in R$, where
$\delta_i$ is a derivation on $R$.\item $S_ir=\sigma_i(r)S_i$ for any $1\leq i\leq m$ and $r\in R$, where
 $\sigma_i$ is a $k$-algebra automorphism on the subalgebra $R[D_1,\ldots,D_n]\subseteq A$ such that $\sigma_i|_R\in \Aut(R)$ and
$\sigma_i(D_j)=\sum_{l=1}^n a_{ijl}D_l,\ \ a_{ijl}\in R.$
\item $S_iS_j=S_jS_i$, $ 1\leq i,j\leq m$; \ $D_{i'}D_{j'}=D_{j'}D_{i'}$, $1\leq i',j'\leq n$.
\item $D_iS_j=S_j\sigma_j (D_i)$, $1\leq i\leq n,1\leq j\leq m$.
\item For any $1\leq i,j\leq n$ and $1\leq i',j'\leq m$,
$
\delta_i\circ\delta_j=\delta_j\circ\delta_i,\
\sigma_{i'}\circ\sigma_{j'}=\sigma_{j'}\circ\sigma_{i'}.
$
\end{enumerate}
\end{definition}

\begin{remark}\upshape
In the above definition, both subalgebras $R[D_1,\ldots,D_n]$ and $R[S_1,\ldots,S_m]$ of $A$ are iterated Ore extensions over $R$. But, in general $A$ is not an iterated Ore extension over $R$.
\end{remark}

The class of \dd\ algebras contains several other known classes of algebras, for example, commutative polynomial algebras, quantum planes, and skew polynomial algebras of derivation (or automorphism) type.

\begin{example}\label{example_QuantumPlane}
Let $0\neq q\in k$ and $I_q$ be the two-sided ideal of the free associative algebra $k\langle x,y\rangle$ generated by the element $yx-qxy$. Then the quotient algebra
$$
k_q[x,y]=k\langle x,y\rangle/I_q
$$
is called a \emph{quantum plane} (\cite{kassel1995quantum}, Chapter IV).
It is easy to see that $k_q[x,y]$ is a \dd\ algebra of type $(1,1)$ over $k$.
\end{example}

The following example distinguishes \dd\ algebras from algebras of solvable type \cite{Weispfenning1990non}, PBW extensions \cite{bell1988uniform}, and G-algebras \cite{levandovskyy2003plural}.

\begin{example}\upshape
 Let $A$ be the $k$-algebra  generated by $\{D_1,D_2,S\}$ with defining relations $\mathcal{R}=\{D_2D_1=D_1D_2,D_1S_1=S_1D_2, D_2S_1=S_1D_1\}$. 
Then it is easy to see that $A$ is a \dd\ algebra of type $(1,2)$ over $k$.
However, by the defining relations, $A$ is not an algebra of solvable type \cite{Weispfenning1990non}, PBW extensions \cite{bell1988uniform}, and G-algebras \cite{levandovskyy2003plural}.
\end{example}

{Let $\mathcal{D} = \{D_1, \dots, D_n\}$ and $\mathcal{S} = \{S_1, \dots, S_m\}$.} If $A$ is a \dd\ algebra over $R$ defined as Definition \ref{Def_DD-alg}, we denote $A=R[\s,\D;\sigma,\delta]$.
For $\alpha=(\alpha_1,\dots,\alpha_m)\in \NN^m$ and $r\in R$, we simply write $\sigma^{\alpha}(r)=\sigma_1^{\alpha_1}\cdots \sigma_m^{\alpha_m}(r)$,
$\D^{\alpha}=D_1^{\alpha_1}\cdots D_m^{\alpha_m}$ and
$|\alpha|=\alpha_1+\cdots+\alpha_m$. In particular, $D_i^0=1$, the identity of $R$.
Similarly, we use notations $\delta^{\alpha}(r)$, $S^\beta$ ($\beta\in\NN^n$), and so on.
Then every element in $A$ can be written uniquely in the form: $\sum_{\alpha,\beta} r_{\alpha,\beta}\s^{\alpha}\D^{\beta}$, where $r_{\alpha,\beta}\in R$ and only finitely many $r_{\alpha,\beta}$ are nonzero. 

\

The following example is taken from \cite{mansfield2003elimination} with some modifications. This example shows where the differential difference algebras come from.

\begin{example}\label{example_DDalgebra}
\upshape
Let $M,n,p\in\NN$ and $p\geq1$. Consider the following system, which arises from the calculation of symmetries of discrete systems (cf., \cite{hydon2000symmetries}),
\begin{eqnarray*}
u_{n+M+1}=\omega(n,u_n,u_{n+1},\ldots,u_{n+M});\\
\mathcal{D}_jF(n,u_n,u_{n+1},\ldots,u_{n+M})=0,\ 1\leq j\leq p,
\end{eqnarray*}
where $F$ is the unknown function and $w$ is a given function in the field $\QQ(n,u_n,\ldots, u_{n+M})$ of rational functions over the rational numbers $\QQ$ in indeterminates $n,u_n,u_{n+1},\ldots, u_{n+M}$, such that $\frac{\partial \omega}{\partial u_n}\neq 0$, and $\mathcal{D}_j:T\to T$ is a linear operator of the form
$$
\mathcal{D}_j=\sum_{\alpha=(\alpha_0,\ldots,\alpha_M)\in\NN^M,~\beta\in\NN}c_{\alpha,\beta}\circ s^{\beta}\circ
\frac{\partial^{\alpha_0+\cdots+\alpha_{_M}}}{\partial u_n^{\alpha_0}\cdots\partial u_{n+M}^{\alpha_{_M}}},
$$
where $T=\QQ(n, u_{n+t}:t\in\ZZ)$, \ $c_{\alpha,\beta}\in \QQ(n, u_{n+t}:t\in\ZZ)$ are \emph{multiplication operators} and only finitely many $c_{\alpha,\beta}$ are nonzero, and $s$ is the \emph{shift operator} defined by $s(n)={n+1}$ and $s(u_n)=u_{n+1}$.

A natural approach to deal with this system is to consider those operators $\mathcal{D}_j$ and $s$ as elements of the noncommutative algebra $A$ over $R$
generated by operator variables $\{S,D_n,\ldots,D_{n+M}\}$, where $S$ denotes the shift operator $s$ and $D_{n+t}$ denotes the differential operator $\frac{\partial}{\partial u_{n+t}}$ for $0\leq t\leq M$, subject to the following commutation rules:
\begin{eqnarray*}
D_{n+t}\circ S&=&S\circ D_{n+t-1}+\frac{\partial \omega}{\partial u_{n+t}}\circ S\circ D_{n+M},\ \ 1\leq t\leq M;\\
D_{n}\circ S&=&\frac{\partial \omega}{\partial u_{n}}\circ S\circ D_{n+M}; \\
D_{n+t}\circ D_{n+t'}&=&D_{n+t'}\circ D_{n+t},\ \ \ \ 0\leq t,t'\leq M ;\\
S\circ r&=&s(r)\circ S,\ \ \ \ r\in R;\\
D_{n+t}\circ r&=&r\circ D_{n+t}+\frac{\partial r}{\partial u_{n+t}},\ \ \ \ r\in R, 0\leq t\leq M.
\end{eqnarray*}
Then $A= \QQ(n, u_{n+t}:t\in\ZZ)[S,\D;\sigma,\delta]$ is a \dd\ algebra of type $(1,M+1)$ over $\QQ(n, u_{n+t}:t\in\ZZ)$,
where $\delta_i=\frac{\partial}{\partial u_i}$ for
$n\leq i\leq {n+M}$ and ${\sigma|_R=s},\sigma(D_n)=s^{-1}(\frac{\partial \omega}{\partial u_{n}}) D_{n+M}, $
$\sigma(D_{n+t})= D_{n+t-1}+s^{-1}(\frac{\partial \omega}{\partial u_{n+t}}) D_{n+M}$ for $1\leq t\leq M$.
\end{example}

Note that the difference operator $\Delta$, defined by $\Delta(u_i)=u_{i+1}-u_i$ for $i\in\ZZ$, can be derived from the shift operator:  $\Delta=S-\id$.
So the \dd\ algebra in the above example actually involves both differential and difference operators.

\section{Gelfand-Kirillov dimension of \dd\ algebras}
In this section, we consider the Gelfand-Kirillov dimension of \dd\ algebras.

We first fix some notations.
Let $\GK(A)$ denote the Gelfand-Kirillov dimension of an algebra $A$, $\dim(V)$  denote the dimension of a $k$-vector space $V$, and $\car(T)$ denote the cardinality of a set $T$. Recall that, for $r,s\in\NN$, the binomial coefficient
$
{r\choose s}=\frac{r(r-1)\cdots(r-s+1)}{s(s-1)\cdots1}
$
if $0\leq s\leq r$; and $
{r\choose s}=0
$
if $s< 0$ or $s>r$.

The Gelfand-Kirillov dimension of an Ore extension has been discussed in \cite{Huh-fand_1996}.

\begin{proposition}(\cite{Huh-fand_1996}, Corollary 2.4, cf., \cite{Krause-Lenagan_2000}, Proposition 3.5.)\label{pro_m=0_n=1}
Let $R$ be a $k$-algebra and $A=R[D;\sigma,\delta]$ be an Ore extension. Suppose that, for each finite dimensional subspace $U$ of $R$,
 there exists a \fd\ subspace $V$ of $R$ such that $U\subseteq V$, $\sigma(V)\subseteq V$ and $\delta(V)\subseteq V^p$ for some $p\geq 1$.
Then $\GK(A)=\GK(R)+1$.
\end{proposition}

We want to consider the Gelfand-Kirillov dimension  of \dd\ algebras satisfying ``similar'' conditions as in Proposition \ref{pro_m=0_n=1}.
Our goal is to find a lower bound and an upper bound of the Gelfand-Kirillov dimension of such a \dd\ algebra.

The following proposition gives a general lower bound of the \gk\ of a \dd\ algebra.
\begin{proposition}\label{Pro_lower_bound}
Let $R$ be a $k$-algebra and
$A=R[\s,\D;\sigma,\delta]$ be a \dd\ algebra of type $(m,n)$.
Then $\GK(A)\geq\GK(R)+m+n$.
\end{proposition}

\begin{proof}
Suppose that $V$ is a \fd\ generating subspace of $R$ and $1\in V$.
Then
\begin{center}
 $W=V+\displaystyle\sum_{i=1}^n kD_i+\displaystyle\sum_{j=1}^m kS_j$
\end{center}
 is a \fd\ generating subspace of $A$.
For any $r\in\NN$,
\begin{eqnarray*}
W^{3r}=(V+\displaystyle\sum_{i=1}^n kD_i+\displaystyle\sum_{j=1}^m kS_j)^{3r}
\supseteq \sum_{\substack{\alpha\in \NN^m,\beta\in\NN^n\\0\leq |\alpha|,~ |\beta|\leq r}}V^r\s^{\alpha}\D^{\beta}.
\end{eqnarray*}
For any $k$-basis $U$ of $V^r$, the set
$$\{u\s^{\alpha}\D^{\beta}:u\in U, 0\leq |\alpha|,|\beta|\leq r,\alpha\in \NN^m,\beta\in\NN^n\}$$
 is a $k$-basis of
$\displaystyle\sum_{0\leq |\alpha|,|\beta|\leq r}V^r\s^{\alpha}\D^{\beta}$.
Hence, we have that
\begin{eqnarray*}
\dim(W^{3r})&\geq& \dim\left(\sum_{0\leq |\alpha|,|\beta|\leq r}V^r\s^{\alpha}\D^{\beta}\right)\\
&=& \dim(V^r)\cdot \ \car(\{\alpha:0\leq \alpha_1+\ldots+\alpha_m\leq r\})\\
& &\ \ \ \ \ \ \  \ \ \ \ \ \ \ \ \cdot \ \car(\{\beta:0\leq \beta_1+\ldots+\beta_n\leq r\})\\
&=&\dim(V^r)\cdot {{r+m-1}\choose{m}}\cdot {{r+n-1}\choose{n}},
\end{eqnarray*}
where ${{r+m-1}\choose{m}}$ and ${{r+n-1}\choose{n}}$ are polynomials in $r$ of degree $m$ and $n$ respectively.
Hence,
\begin{eqnarray*}
\GK(A)&\geq&\overline{\lim_{r\to \infty}} \log_r \dim(W^r)
=\overline{\lim_{r\to \infty}} \log_r \dim(W^{3r})\\
&\geq&\overline{\lim_{r\to \infty}} \log_r \left(\dim(V^r)\cdot {{r+m-1}\choose{m}}\cdot {{r+n-1}\choose{n}}\right)\\
&=& \GK(R)+m+n.
\end{eqnarray*}
\end{proof}

In the special case $R=k$, the equality in the above proposition holds, i.e., we have the following proposition,
which indicates that the lower bound of $\GK(A)$ obtained in Proposition \ref{Pro_lower_bound} can not be sharpened any more.
\begin{proposition}\label{Pro_R=k}
Let $A=k[\s,\D;\sigma,\delta]$ be a \dd\ algebra of type $(m,n)$ over $k$.
Then $\GK(A)=m+n$.
\end{proposition}
\begin{proof}
Let $V=k+\displaystyle\sum_{i=1}^n kD_i+\displaystyle\sum_{j=1}^m kS_j$. Then $V$ is a \fd\ generating subspace of $A$.
For any $r\in\NN$,
$$
V^r =  (k+\displaystyle\sum_{i=1}^n kD_i+\displaystyle\sum_{j=1}^m kS_j)^r
 \subseteq \sum_{0\leq |\alpha|,|\beta|\leq r}k\s^{\alpha}\D^{\beta},
$$
where the last inclusion holds since
$$
D^{\beta}S^{\alpha}\in \displaystyle\sum_{\beta'\in\NN^n,|\beta'|=|\beta|}kS^{\alpha}D^{\beta'},\
\ \alpha\in\NN^m,\ \beta\in\NN^n.
$$
So, $\dim(V^r)\leq {{r+m-1}\choose{m}}\cdot {{r+n-1}\choose{n}}$. Hence, $\GK(A)\leq m+n$ and thus by Proposition \ref{Pro_lower_bound}
$\GK(A)= m+n$.
\end{proof}

Now let us turn to  upper bounds for $\GK(A)$. First we consider the case when $R$ is finitely generated.

\begin{lemma}\label{lemma_GK(A)}
Let $R$ be a $k$-algebra with a finite dimensional generating subspace $V$, and let
$A=R[\s,\D;\sigma,\delta]$ be a \dd\ algebra of type $(m,n)$. Suppose that $\sigma_i(V)\subseteq V$ for $1\leq i\leq m$. Then
$\GK(A)\leq 2\GK(R)+m+n.$

Furthermore, if, for all $1\leq i\leq m$ and $1\leq j\leq n$, $\sigma_i(D_j)$ is contained in the vector space over $k$ generated
by $\{D_1,\ldots,D_n\}$, then
$\GK(A)=\GK(R)+m+n.$
\end{lemma}

\begin{proof}
Since $V$ is a generating subspace, there exists $p\geq 1$ such that
$$
a_{ijl}\in V^p,\ \delta_i(V)\subseteq V^p,\ 1\leq i,l\leq n,\ 1\leq j\leq m,
$$
where the $a_{ijl}$ are the coefficients that appear in Definition \ref{Def_DD-alg}.
Then
$$
\delta_i(V^t)\subseteq V^{p+t},\ \sigma_j(V^t)\subseteq V^t, \ 1\leq i\leq n,\ 1\leq j\leq m,\ t\geq1.
$$
So, eventually replacing $V$ by $V^p$ if necessary, we may assume that
$$1\in V, \delta_i(V)\subseteq V^2, \sigma_j(V)\subseteq V,\ 1\leq i\leq n,\ 1\leq j\leq m.$$
Let $X=\displaystyle\sum_{i=1}^n kD_i$,
$Y=\displaystyle\sum_{j=1}^m kS_j$ and $W=V+X+Y$.
Then $W$ is a generating subspace of $A$.

In order to finish the first statement of the this lemma, we have to  prove the following three lemmas first.
\begin{lemma}\label{Lemma_V_X_Y}
For any integer $s\geq1$,
\begin{enumerate}[(i).]
\item $XY\subseteq VYX,\ XV\subseteq VX+V^2,\ YV=VY$.
\item $X^sV\subseteq \displaystyle\sum_{i=0}^s V^{i+1}X^{s-i}$.
\item $X^sY\subseteq \displaystyle \sum_{i=0}^{s-1}V^{s+i}YX^{s-i}.$
\end{enumerate}
\end{lemma}
\begin{proof}
(i). It follows easily by definition.
(ii). (By induction on $s$.) If $s=1$, then we have $XV\subseteq VX+V^2$ by the commutation rules of \dd\ algebras. Suppose that
$X^rV\subseteq \displaystyle\sum_{i=0}^r V^{i+1}X^{r-i}$ for $1\leq r\leq s$.
Then
\begin{eqnarray*}
X^{s+1}V&\subseteq& X^s(VX+V^2)
\subseteq\sum_{i=0}^s V^{i+1}X^{s-i+1}+\sum_{i=0}^s V^{i+1}X^{s-i}V\\
&\subseteq&\sum_{i=0}^s V^{i+1}X^{s-i+1}+\sum_{i=0}^s V^{i+1}\sum_{j=0}^{s-i}V^{j+1}X^{s-i-j}\\
&=&\sum_{i=0}^s V^{i+1}X^{s-i+1}+\sum_{i=0}^s \sum_{j=0}^{s-i}V^{i+j+2}X^{s-i-j}\ \ \\
&=&\sum_{i=0}^s V^{i+1}X^{s-i+1}+\sum_{i=0}^s \sum_{l=i+1}^{s+1}V^{l+1}X^{s-l+1}\hspace{20mm} (l:=i+j+1)\\
&=&\sum_{i=0}^s V^{i+1}X^{s-i+1}+ \sum_{l=1}^{s+1}V^{l+1}X^{s-l+1}\\
&=&\sum_{i=0}^{s+1} V^{i+1}X^{s-i+1}.
\end{eqnarray*}
Thus (ii) holds for all $s\geq 1$.

(iii). (By induction on $s$.)
If $s=1$, then we have $V^1YX^1\supseteq XY$ by (i), and thus (iii) holds.
Suppose that $X^rY\subseteq\displaystyle  \sum_{i=0}^{r-1}V^{r+i}YX^{r-i}$ for all $1\leq r\leq s$.
Then,
\begin{eqnarray*}
X^{s+1}Y&\subseteq & X^{s}(VYX)
\subseteq \sum_{i=0}^{s} V^{i+1}X^{s-i}YX\ \ \ (\mbox{by (ii)})\\
&= &\sum_{i=0}^{s-1} V^{i+1}X^{s-i}YX+V^{s+1}YX\\
&\subseteq &\sum_{i=0}^{s-1}V^{i+1}\sum_{j=0}^{s-i-1}V^{s-i+j}YX^{s-i-j+1}+V^{s+1}YX\\
&=&\sum_{i=0}^{s-1}  \sum_{j=0}^{s-i-1}V^{s+j+1}YX^{s-i-j+1}+V^{s+1}YX\\
&=&\sum_{i=0}^{s-1}  \sum_{l=i}^{s-1}V^{s+l-i+1}YX^{s-l+1}+V^{s+1}YX\ \ \ (l:=i+j)\\
&\subseteq&\sum_{i=0}^{s-1}  \sum_{l=0}^{s-1}V^{s+l+1}YX^{s-l+1}+V^{s+1}YX \\
&=&\sum_{l=0}^{s-1}V^{s+l+1}YX^{s-l+1}+V^{s+1}YX
\subseteq\sum_{l=0}^{s}V^{s+l+1}YX^{s-l+1}.\\
\end{eqnarray*}
Hence (iii) holds.
\end{proof}

\begin{lemma}\label{Lemma_W^r+1}
For all $r\geq 1$, $W^r\subseteq\displaystyle \sum_{i=0}^{r}\sum_{j=0}^{r-i}V^{2r^2}Y^iX^{j}$.
\end{lemma}
\begin{proof}
(By induction on $r$.) If $r=1$, then the right hand side of the inclusion is $V^2+V^2X+V^2Y\supseteq W$.
Suppose the statement is true for $r\geq1$. Then, by induction hypothesis,
\begin{eqnarray*}
&&W^{r+1} \subseteq \sum_{i=0}^{r}\sum_{j=0}^{r-i}V^{2r^2}Y^iX^{j}(V+X+Y)\hspace{6mm} (\mbox{induction hypothesis})  \\
&= &\sum_{i=0}^{r}\sum_{j=0}^{r-i}V^{2r^2}Y^iX^{j}V+\sum_{i=0}^{r}\sum_{j=0}^{r-i}V^{2r^2}Y^iX^{j+1}
+\sum_{i=0}^{r}\sum_{j=0}^{r-i}V^{2r^2}Y^iX^{j}Y\\
&\subseteq &\sum_{i=0}^{r}\sum_{j=0}^{r-i}V^{2r^2}Y^i\sum_{l=0}^{j}V^{l+1}X^{j-l}\hspace{29mm} (\mbox{by Lemma \ref{Lemma_V_X_Y}(ii)})\\
&&\ \ \ +\sum_{i=0}^{r}\sum_{j=0}^{r-i}V^{2r^2}Y^iX^{j+1}\\
&&\ \ \ \ \ \ +\sum_{i=0}^{r}\sum_{j=0}^{r-i}V^{2r^2}Y^i\sum_{l=0}^{j-1}V^{j+l}YX^{j-l}\hspace{15mm} (\mbox{by Lemma \ref{Lemma_V_X_Y}(iii)})\\
&=&\sum_{i=0}^{r}\sum_{j=0}^{r-i}\sum_{l=0}^{j}V^{2r^2+l+1}Y^iX^{j-l}
+\sum_{i=0}^{r}\sum_{j=0}^{r-i}V^{2r^2}Y^iX^{j+1}\\
&&\hspace*{60mm}+\sum_{i=0}^{r}\sum_{j=0}^{r-i}\sum_{l=0}^{j-1}V^{2r^2+j+l}Y^{i+1}X^{j-l}\\
&=&\sum_{i=0}^{r}\sum_{j=0}^{r-i}\sum_{p=0}^{j}V^{2r^2+j-p+1}Y^iX^{p}\hspace{35mm} (p:=j-l)\\
&&\ \ \ +\sum_{i=0}^{r}\sum_{j=1}^{r-i+1}V^{2r^2}Y^iX^{j}\hspace{43mm}(\mbox{shift index $j$})\\
&&\ \ \ \ \ \ +\sum_{i=1}^{r+1}\sum_{j=0}^{r-i+1}\sum_{p=1}^{j}V^{2r^2+2j-p}Y^{i}X^{p}\hspace{15mm} (\mbox{shift $i$ and } p:=j-l)\\
&\subseteq&\sum_{i=0}^{r}\sum_{j=0}^{r-i}\sum_{p=0}^{r-i}V^{2r^2+r-p+1}Y^iX^{p}
+\sum_{i=0}^{r}\sum_{j=1}^{r-i+1}V^{2r^2}Y^iX^{j}\\
&&\ \ \ \ \ \ +\sum_{i=1}^{r+1}\sum_{j=0}^{r-i+1}\sum_{p=1}^{r-i+1}V^{2r^2+2r-p}Y^{i}X^{p}\\
&\subseteq&\sum_{i=0}^{r}\sum_{p=0}^{r-i}V^{2r^2+r-p+1}Y^iX^{p}
+\sum_{i=0}^{r}\sum_{j=1}^{r-i+1}V^{2r^2}Y^iX^{j}
+\sum_{i=1}^{r+1} \sum_{p=1}^{r-i+1}V^{2r^2+2r-p}Y^{i}X^{p}\\
&\subseteq &\sum_{i=0}^{r+1}\sum_{j=0}^{r-i+1}V^{2(r+1)^2}Y^iX^{j}.
\end{eqnarray*}
Therefore, $W^r\subseteq\displaystyle \sum_{i=0}^{r}\sum_{j=0}^{r-i}V^{2r^2}Y^iX^{j}$ for any $r\geq 1$.
\end{proof}

\begin{lemma}\label{Lemma_growth^2}
Let $f:\mathbb{N}\to \RR$ be an increasing and positive valued function, and $p>1$.
Then $$
\overline{\lim_{n\to\infty}}\log_nf(pn^2)\leq 2\overline{\lim_{n\to\infty}}\log_nf(n).
$$
\end{lemma}

\begin{proof}
Let $d= \displaystyle\overline{\lim_{n\to\infty}}\log_nf(n)$.
By Lemma 2.1 of \cite{Krause-Lenagan_2000},
 $$d=\inf\{\rho\in \RR: f(n)\leq n^{\rho} \mbox{ for almost all } n\in \NN\}.$$
Hence, for any $\varepsilon>0$,
$f(n)< n^{d+\varepsilon}$ for almost all $n$. So
$$
f(pn^2)<(pn^2)^{d+\varepsilon}=p^{d+\varepsilon}n^{2d+2\varepsilon}<n^{2d+3\varepsilon} \ \mbox{for almost all } n.
$$
Therefore,
$
\displaystyle\overline{\lim_{n\to\infty}}\log_nf(pn^2)\leq 2d=2\overline{\lim_{n\to\infty}}\log_nf(n).
$
\end{proof}

Now let us return to the proof of Lemma \ref{lemma_GK(A)}.
Let $f(r)=\dim(V^r)$ for $r\in \NN$.
Then
\begin{eqnarray*}
\GK(A)&=&\overline{\lim_{r\to\infty}}\log_r\dim(W^r)
\leq \overline{\lim_{r\to\infty}}\log_r\dim\left(\sum_{i=0}^{r}\sum_{j=0}^{r-i}V^{2r^2}Y^iX^{j}\right)\\
&=& \overline{\lim_{r\to\infty}}\log_r\left(\dim(V^{2r^2})\sum_{i=0}^{r}\sum_{j=0}^{r-i}\dim(Y^i)\dim(X^{j})
\right)\\
&=&\overline{\lim_{r\to\infty}}\log_r\left(f(2r^2)\sum_{i=0}^{r}\sum_{j=0}^{r-i}{{i+m-1}\choose {m-1}}{{j+n-1}\choose{n-1}}
\right)\\
&\leq&\overline{\lim_{r\to\infty}}\log_rf(2r^2)+\overline{\lim_{r\to\infty}}\log_r\sum_{i=0}^{r}{{i+m-1}\choose {m-1}}\\
&&\hspace{45mm}+
\overline{\lim_{r\to\infty}}\log_r\sum_{j=0}^{r}{{j+n-1}\choose{n-1}}
\\
&\leq&2\GK(R)+m+n
\end{eqnarray*}
where the last inequality holds because of Lemma \ref{Lemma_growth^2} and the fact that
if $p(i)$ is a polynomial in $i$ of degree $s$ then $\displaystyle\sum_{i=0}^rp(i)$ is a polynomial in $r$ of degree $s+1$.

\

For the second statement of Lemma \ref{lemma_GK(A)}, we first note that, for all $1\leq i\leq m$ and $1\leq j\leq n$, if $\sigma_i(D_j)\in X$ then $XY=YX$.
Then, under the given assumptions, we have the following lemma.
\begin{lemma}\label{lemma_W^r_if_XY=YX}
For all $r\geq 1$, $\displaystyle W^r\subseteq \sum_{i=0}^{r}\sum_{j=0}^{r-i}V^{2r-i-j}Y^iX^{j}.$
\end{lemma}
\begin{proof}
(By induction on $r$.)
It is easy to check that the inclusion is true for $r=1$. Suppose $W^r\subseteq\displaystyle \sum_{i=0}^{r}\sum_{j=0}^{r-i}V^{2r-i-j}Y^iX^{j}$ for $r\geq 1$.
Then
\begin{eqnarray*}
W^{r+1}&\subseteq&\left(\sum_{i=0}^{r}\sum_{j=0}^{r-i}V^{2r-i-j}Y^iX^{j}\right)(V+X+Y)\hspace{6mm}\mbox{(by induction hypothesis)}\\
&=&\sum_{i=0}^{r}\sum_{j=0}^{r-i}\left(V^{2r-i-j}Y^iX^{j}V+V^{2r-i-j}Y^iX^{j+1}+V^{2r-i-j}Y^iX^{j}Y\right)\\
&\subseteq&\sum_{i=0}^{r}\sum_{j=0}^{r-i}V^{2r-i-j}Y^i\left(\sum_{l=0}^jV^{l+1}X^{j-l}\right)\hspace{18mm}\mbox{(by Lemma \ref{Lemma_V_X_Y}(ii))}\\
&&                \hspace{25mm}+\sum_{i=0}^{r}\sum_{j=0}^{r-i}V^{2r-i-j}Y^iX^{j+1}+\sum_{i=0}^{r}\sum_{j=0}^{r-i}V^{2r-i-j}Y^{i+1}X^{j}\\
&=&\sum_{i=0}^{r}\sum_{j=0}^{r-i}\sum_{t=0}^jV^{2r-i-t+1}Y^iX^{t}
\hspace{40mm}(t:=j-l)\\
&&                \hspace{20mm}+\sum_{i=0}^{r}\sum_{j=1}^{r+1-i}V^{2r-i-j+1}Y^iX^{j}\hspace{20mm}\mbox{(shift index $j$ )}\\
&&\hspace{40mm}+\sum_{i=1}^{r+1}\sum_{j=0}^{r-i}V^{2r-i-j+1}Y^{i}X^{j} \hspace{3mm}\mbox{(shift index $i$ )}\\
&\subseteq&\sum_{i=0}^{r+1}\sum_{j=0}^{r+1-i}V^{2(r+1)-i-j}Y^iX^{j}
\end{eqnarray*}
That proves our lemma.
\end{proof}

By the above lemma, $W^r\subseteq \displaystyle \sum_{i=0}^{r}\sum_{j=0}^{r-i}V^{2r-i-j}Y^iX^{j}\subseteq
\sum_{i=0}^{r}\sum_{j=0}^{r-i}V^{2r}Y^iX^{j}$. Hence, by a similar argument as we used in the proof of the first statement of
Lemma \ref{lemma_GK(A)}, we have that
$\GK(A)\leq \GK(R)+m+n$. Thus, by Proposition \ref{Pro_lower_bound}, $\GK(A)=\GK(R)+m+n$.
\end{proof}

Now we are in a position to state our main theorem.
\begin{theorem}\label{thm_GK(A)}
Let $R$ be a $k$-algebra  and
$A=R[\s,\D;\sigma,\delta]$ be a differential difference algebra of type $(m,n)$.
Suppose that for any \fd\ subspace $U$ of $R$ there exist a \fd\ subspace $V$ of $R$ and an integer $p\geq1$ such that
$$
U\subseteq V,\ \sigma_i(V)\subseteq V,\ \delta_j(V)\subseteq V^p,\ \ 1\leq i\leq m, 1\leq j\leq n.
$$
Then
$$\GK(R)+m+n\leq \GK(A)\leq 2\GK(R)+m+n.$$
Furthermore, if, for all $1\leq i\leq m$ and $1\leq j\leq n$, $\sigma_i(D_j)$ is contained in the vector space over $k$ generated
by $\{D_1,\ldots,D_n\}$, then
$$\GK(A)=\GK(R)+m+n.$$
\end{theorem}
\begin{proof}
Let $W$ be a \fd\ subspace of $A$ with a $k$-basis $w_1,\ldots,w_q, q\in\NN$.
Note that each $w_i, 1\leq i\leq q$, is a polynomial in $D_1,\ldots,D_n,S_1,\ldots,S_m$ with coefficients in $R$.
Let $U$ be the subspace of $R$ spanned by all the coefficients (in $R$) of $w_1,\ldots,w_q$ and all $a_{ijl}$
(defined in Definition \ref{Def_DD-alg}),
$1\leq i,l\leq n,1\leq j\leq m$.
Then $U$ is \fd\ and hence there exist a \fd\ subspace $V$ of $R$ and an integer $p\geq1$ such that
$U\subseteq V,\ \sigma_i(V)\subseteq V,\ \delta_j(V)\subseteq V^p$ for $1\leq i\leq m, 1\leq j\leq n$.
Let $B$ be the subalgebra of $R$ generated by $V$.
Then $\sigma_i(B)\subseteq B$, $\sigma_i(D_l)\in B[\D;\delta]$ and $\delta_j(B)\subseteq B$ for $1\leq i\leq m, 1\leq j,l\leq n$. That is, $A'=B[\s,\D;\sigma,\delta]$ is a \dd\ algebra satisfying the conditions of Lemma \ref{lemma_GK(A)}. Note that $W\subseteq A'$. So, by Lemma \ref{lemma_GK(A)}, we have
$$
\overline{\lim_{r\to\infty}}\log_r\dim(W^r)\leq \GK(A')\leq2\GK(B)+m+n\leq 2\GK(R)+m+n.
$$
Thus $\GK(A)\leq 2\GK(R)+m+n$ since $W$ is arbitrary. Therefore, by Lemma \ref{Pro_lower_bound},
$
\GK(R)+m+n\leq \GK(A)\leq 2\GK(R)+m+n.
$
That completes our proof of the first statement.

The second statement follows similarly by using the second part of Lemma \ref{lemma_GK(A)}.
\end{proof}

Immediately from Theorem \ref{thm_GK(A)}, we have the following corollaries.

\begin{corollary}\label{coro_GK(quantum)}
The quantum plane $k_q[x,y]$ has Gelfand-Kirillov dimension $2$.
\end{corollary}

Recall that an algebra $A$ is called \emph{locally finite dimensional} if every finitely generated subalgebra of $A$ is finite dimensional.
\begin{corollary}\label{coro_GK(R)=0}
Let $R$ be a $k$-algebra  and
$A=R[\s,\D;\sigma,\delta]$ be a \dd\ algebra of type $(m,n)$ satisfying the conditions of the first statement of Theorem \ref{thm_GK(A)}.
\begin{enumerate}[(i).]
\item If $R$ is locally \fd, then $\GK(A)=m+n$.
\item If $\GK(R)<\infty$, then $\GK(A)<\infty$.
\end{enumerate}
\end{corollary}
\begin{proof}
(i). It follows from the fact that $\GK(R)=0$ if and only if $R$ is locally \fd.

(ii). This is clear.
\end{proof}

Note that if we set $R=k$ in Theorem \ref{thm_GK(A)}, then the conditions of the theorem are satisfied. Thus Theorem \ref{thm_GK(A)} implies Proposition \ref{Pro_R=k}.

The following example shows that the upper bound of $\GK(A)$ stated in Theorem \ref{thm_GK(A)} is the ``best'' one under the given conditions.
\begin{example}
Let $A$ be the $k$-algebra generated by $\{z,z^{-1},D,S\}$ with defining relations
\begin{eqnarray*}
\mathcal{R}=\{zz^{-1}=1,\ \ z^{-1}z=1, &Dz=zD,&Sz=zS, \\
Dz^{-1}=z^{-1}D,&Sz^{-1}=z^{-1}S,&DS=zSD\}.
\end{eqnarray*}
Let $R=k[z,z^{-1}]$ be the algebra of Laurent polynomials over $k$, and let
$\sigma$ be the \auto\ of the algebra $R[D]\subseteq A$ defined by
$$
\sigma(\sum_{i=0}^lc_iD^i)=\sum_{i=0}^lc_i(zD)^i, \ \ l\geq0, c_i\in R \mbox{ for } 0\leq i\leq l.
$$
Then $
A=R[S,D;\sigma,0]
$
is a \dd\ algebra of type $(1,1)$ and $$\GK(A)=2\GK(R)+1+1=4.$$
\end{example}
\begin{proof}
It is easy to see that $A$ can be thought of as an iterated Ore extension over $R$: $A=R[S;\id, 0][D;\sigma',0]$ where $\sigma'$ is the automorphism over $R[S]$ defined by 
$$
\sigma'(\sum_{i=0}^lc_iS^i)=\sum_{i=0}^lc_i(zS)^i, \ \ l\geq0, c_i\in R \mbox{ for } 0\leq i\leq l.
$$
Hence $\{S^iD^j:i,j\in\NN\}$ forms an $R$-basis of $A$. Thus $A$ is a \dd\ algebra.

Note that the restriction of $\sigma$ on $R$ is the identity automorphism of $R$.
It is clear that $A$ satisfies all conditions of Theorem \ref{thm_GK(A)}. So, by Theorem \ref{thm_GK(A)}, $\GK(A)\leq 2\GK(R)+2$.
Since $\GK(R)=1$ (see, for example, Corollary 8.2.15 of \cite{McConnell-Robson2001}), $\GK(A)\leq 4$.

Note that $DS=S\sigma(D)=SzD=zSD$. Then one can prove that $D^jS=z^{j}S D^j$ by induction on $j$,
and then that $D^jS^i=z^{ij}S^i D^j$ by induction on $i$.
Now we claim that
$$
B:=\{z^lS^iD^j: 0\leq i+j\leq r, 0\leq l\leq ij\}\subseteq W^{r},\ r\geq1,
$$
where $W=k+kz+kz^{-1}+kD+kS$ is a generating subspace of $A$.
Suppose $r\geq 1, 0\leq i+j\leq r, 0\leq l\leq ij$ and write $l=qj+p$ with $0\leq q\leq i, 0\leq p<j$.
If $q=i$, then $p=0,l=ij$ and $z^lS^iD^j=z^{ij}S^iD^j=D^j S^i\in W^r$. If $q<i$,
then
\begin{eqnarray*}
S^{i-q-1}D^pSD^{j-p}S^q=z^{p}S^{i-q-1}SD^pD^{j-p}S^q\\
=z^{p}S^{i-q}D^{j}S^q=z^{p+qj}S^{i-q}S^q D^{j}=z^lS^iD^j.
\end{eqnarray*}
Since $(i-q-1)+p+1+(j-p)+q=i+j\leq r$, $z^lS^iD^j =S^{i-q-1}D^pSD^{j-p}S^q\in W^r$.
Thus our claim holds.

It is clear that $B$ is $k$-linearly independent. Then by our claim,
$$
\dim(W^{r})\geq \car(B)=\sum_{i=0}^r\sum_{j=0}^r(ij+1)
=\frac{1}{4}r^4+\frac{1}{2}r^3+\frac{5}{4}r^2+2r+1.
$$
Thus, $\GK(A)\geq \displaystyle\overline{\lim_{r\to\infty}}\log_r\dim(W^r)\geq 4$.
Therefore, $\GK(A)=4$.
\end{proof}

\textbf{Acknowledgements.} The authors would like to thank Prof. G\"unter Krause for his careful reading of the manuscript and valuable comments. This work is supported in part by  the National Sciences and Engineering
Research Council of Canada.


\end{document}